\documentclass[review]{elsarticle}

\usepackage{lineno,hyperref}

\usepackage{amssymb}
\usepackage{amsthm}
\usepackage{latexsym}
\usepackage{amsmath}
\usepackage{xcolor}
\usepackage{mathrsfs}
\usepackage{mathtools}
\usepackage{graphicx}
\usepackage{tabularx}
\usepackage{url}
\usepackage{tikz}
\usepackage[title]{appendix}
\usepackage[linesnumbered,ruled]{algorithm2e}
\newcommand{\col}{\text{col}}
\newcommand{\row}{\text{row}}
\usepackage{upgreek}
\usepackage{microtype}
\usepackage{stmaryrd}
\usepackage{enumerate}
\usepackage{comment}
\usepackage{times}
\usepackage{t1enc}
\usepackage{upgreek}
\usepackage{xcolor}
\usepackage{pdflscape}

\newtheorem{Corollary}{Corollary}
\newtheorem{Theorem}{Theorem}
\newtheorem{Lemma}{Lemma}

\newtheorem{Definition}{Definition}
\newtheorem{Proposition}{Proposition}
\newtheorem{Example}{Example}
\newtheorem*{Remarks*}{Remarks}

\begin{document}

\begin{frontmatter}

\title{Two-dimensional Fibonacci Words: Tandem Repeats and Factor Complexity}


\author[mymainaddress]{Sivasankar M}
\corref{mycorrespondingauthor}
\cortext[mycorrespondingauthor]{Corresponding author}
\ead{ma16d028@smail.iitm.ac.in}

\author[mymainaddress]{Rama R}
\ead{ramar@iitm.ac.in}

\address[mymainaddress]{Department of Mathematics, Indian Institute of Technology Madras, Chennai, India}

\begin{abstract}
 If $x$ is a non-empty string then the repetition $xx$ is called a tandem repeat. Similarly, a tandem in a two dimensional array $X$ is a configuration consisting of a same primitive block $W$ that touch each other with one side or corner. In \cite{Apostolico:2000}, Apostolico and Brimkov have proved various bounds for the number of tandems in a two dimensional word of size $m \times n$. Of the two types of tandems considered therein, they also proved that, for one type, the number of occurrences in  an $m \times n$ Fibonacci array attained the general upper bound, $\mathcal{O}(m^{2}n \hspace{0.1cm} \mbox{log} \hspace{0.1cm} n)$. In this paper, we derive an expression for the exact number of tandems in a given finite Fibonacci array $f_{m,n}$. As a required result, we derive the factor complexities of $f_{m,n}$, $m,n \ge 0$ and that of the infinite Fibonacci word $f_{\infty, \infty}$. Generations of $f_{\infty, \infty}$ and $f_{m,n}$, for any given $m,n \ge 1$ using a two-dimensional homomorphism is also achieved.
\end{abstract}


\end{frontmatter}


\section{Introduction}
The field of combinatorics of words produces many path breaking results which are directly used in computer science, molecular biology and particle physics. To be more specific, we can mention pattern recognition, image processing techniques in computer science, DNA computing in molecular biology and crystallography in physics \cite{Lothaire:1997, PaunDNA:1998, Ramona:2008}. In these research directions, similar to palindromes, another interesting and important structure is repetitive substrings.  If $x$ is a non-empty string then the repetition $xx$ is called a tandem repeat or a square. Axel and Thue's work on square free words and avoidable patterns, provided more insights in to tandems \cite{Berstel:1992}. In DNAs, when a pattern of one or more nucleotides is repeated, with the repetitions directly adjacent to each other, we say a tandem has occurred. More information about tandem repeats can help in determining inherited traits of an individual and are useful in genealogical DNA tests. The capacity and expressiveness of genomic tandem duplication is explored in \cite{Jain:2017}. In computer science engineering, tandem repeats are used in compression algorithms \cite{Salomon:1998}.  

A natural extension of one dimensional (hereafter, sometimes denoted by $1D$) words is to two dimensions. A two dimensional (hereafter, sometimes denoted by $2D$) word of size $m \times n$ is a rectangular arrangement of symbols from an alphabet, in $m$ rows and $n$ columns. A tandem in a two dimensional array $X$ is a configuration consisting of a same primitive block $W$ that touch each other with one side or corner. Basically, tandem repeats in one dimensional strings are detected using suffix trees \cite{Gusfield:2004, Crochemore:2009}. It is natural to expect more amount of work (compared to the one dimensional setup) to detect a tandem in two dimensions. Algorithms to locate periods,palindromes, runs etc. also become more complex in a multi dimensional setup \cite{Amir:2020a, Charalampopoulos:2020, Amir:2020b, Amir:2022}. 

In \cite{Apostolico:2000},the authors have proved various bounds for the number of tandems in a two dimensional word of size $m \times n$. In fact they have considered two types of tandems. They have introduced a sequence of $2D$ words called $2D$ Fibonacci arrays and proved that the number of occurrences of one of the tandem type attains the general upper bound of $\mathcal{O}(m^{2}n \hspace{0.1cm} \mbox{log} \hspace{0.1cm} n)$  in  an $m \times n$ Fibonacci array. Like Fibonacci numbers, Fibonacci words/arrays  are very exciting. For a detailed study of Fibonacci words \cite{Berstel:1986,Yu:2000} can be referred. Continuing the work done in \cite{Apostolico:2000} on Fibonacci arrays, we count the exact number of tandems in a given Fibonacci array $f_{m,n}$.

The main contributions of this paper are:
\begin{itemize}
    \item[-]  The exact number of tandems occurring (with repetition and without repetition) in a given $f_{m,n}$ are counted. 
    \item[-] Factor complexity of finite Fibonacci words are obtained.
    \item[-] A two dimensional morphism generating the 2D infinite Fibonacci word $f_{\infty, \infty}$ is developed. 
    \item[-] A Deterministic Finite state Automaton with Output (DFAO) is constructed for $f_{\infty, \infty}$.  
\end{itemize}

This paper is organised as follows. Section 2 has the prerequisites for understanding the later sections. Section 3 explains the types of tandems that occur in a $2D$ word. Section 4 counts the number of Tandems (repetitions included) in a given $f_{m,n}$. In Section 5 we find the factor complexity of the one-dimensional Fibonacci word $f_{n}$ which is used in Section 6 to find the number of distinct tandems in $f_{m,n}$.  In Section 7 the $2D$ morphism generating $f_{\infty, \infty}$ is developed. In Section 8 the factor complexities of $f_{m,n}$ and $f_{\infty, \infty}$ are found out. Section 9 discusses the DFAO generating $f_{\infty, \infty}$.

\section{Preliminaries}

\subsection{Words Over an Alphabet}
A finite non-empty set of symbols is called an alphabet and is denoted by $\Sigma$. A word $w = x_1x_2\cdots x_n$ of length $n$, is the juxtaposition (familiarly known as concatenation) of symbols $x_1, x_2, \ldots, x_n$ taken from $\Sigma$. The length of the word $w$, that is the number of symbols in $w$, is denoted by $|w|$. The set of all words over $\Sigma$ including the empty word $\lambda$, is denoted by $\Sigma^{*}$, whereas $\Sigma^{+}$ denotes the set of all non-empty words over $\Sigma$. In fact, $\Sigma^*$ is a free monoid under the operation concatenation. A word $x \in \Sigma^*$ is a factor of another word $w \in \Sigma^*$ if $w=uxv$ for some  $u,v \in \Sigma^*$. A word $x \in  \Sigma^{*}$ is a prefix (suffix, respectively) of the word $w$ if $w = xy$ ($w=yx$, respectively) for some $y \in \Sigma^{*}$. The reversal of $w=x_{1}x_{2} \cdots x_{n}$ is defined to be the string $w^{R}=x_{n} \cdots x_{2} x_{1}$. A word $w$ is said to be a palindrome or a one-dimensional palindrome if $w=w^{R}$. Powers $w^k$, $k\ge0$  of a word $w$ are obtained by concatenating $w$ with itself, $k$ number of times. A word $w$ is said to be primitive if $w=u^{n}$ implies $n=1$ and $w=u$. A square in a word $w$ is a subword of $w$, which is of the form $xx$, $x\in \Sigma^+$. For a more elaborate study of formal language theory and combinatorics on words, the reader is referred to \cite{Kamala:2009,Lothaire:2002}.

\subsection{Two-dimensional Words}

A subset of $\Sigma^*$ is called a language. Studying the type and the grammar of the words in a language constitute the formal language theory where as analysing number of squares, palindromes etc. is combinatorics on words. Extending formal language theory and combinatorics of words to two dimensions is a challenging task. The difficulty arises due to the presence of two directions. Below a short introduction to two-dimensional languages is given. Interested reader can refer \cite{Giammarresi:1997} for further concepts.

\begin{Definition}\cite{Giammarresi:1997}
Let $\Sigma$ be an alphabet. An array $($also called a picture or two-dimensional word$)$ $u=[u_{i,j}]_{1 \leq i \leq m,1 \leq j \leq n}$ of size $(m,n)$ over  $\Sigma$ is a two-dimensional rectangular finite arrangement of letters:
$$u=
\begin{matrix}
u_{1,1} & u_{1,2} & \cdots & u_{1,n-1} & u_{1,n} \\
u_{2,1} & u_{2,2} & \cdots & u_{2,n-1} & u_{2,n} \\
\vdots & \vdots & \ddots & \vdots & \vdots \\
u_{m-1,1} & u_{m-1,2} & \cdots & u_{m-1,n-1} & u_{m-1,n} \\
u_{m,1} & u_{m,2} & \cdots & u_{m,n-1} & u_{m,n} \\
\end{matrix}$$
\end{Definition}

The number of rows and columns of $u$ are denoted, respectively, by $|u|_{\text{row}}$ and $|u|_{\text{col}}$. The array of size $(0,0)$ denoted by $\Lambda$ is the empty array. The arrays of sizes $(m,0)$ and $(0,m)$ for $m>0$ are not defined. It is noted that some authors consider these arrays also as the empty array. The set of all arrays over $\Sigma$ including the empty array, $\Lambda$, is denoted by $\Sigma^{**}$, whereas $\Sigma^{++}$ is the set of all non-empty arrays over $\Sigma$.

To locate any position or region in an array, we require a reference system \cite{Anselmo:2014}. Given an array $u$, the set of coordinates $\{1,2,\ldots,|u|_{\text{row}}\} \times \{1,2,\ldots,|u|_{\text{col}}\}$ is referred to as the domain of $u$. We also use t, b, l, r (the initials of the words top, bottom, left, right, respectively) to detect the sides or boundaries of $u$. A subdomain or subarray of an array $u$, denoted by $u[(i,j),(i',j')]$, is the portion of $u$ located in the region $\{i,i+1,\ldots,i'\} \times \{j,j+1,\ldots,j'\}$, where $1 \leq i \leq |u|_{\row},1 \leq j \leq |u|_{\col}$. Below we state the concatenation operation between two arrays and further definitions associated with 2D arrays.

\begin{Definition}
\cite{Giammarresi:1997} Let $u,v$ be arrays over $\Sigma$, of sizes $(m_{1},n_{1})$ and $(m_{2},n_{2})$, respectively with $m_{1},n_{1},m_{2},n_{2}>0$. Then,
\begin{enumerate}
\item The column concatenation of $u$ and $v$, denoted by $\obar$, is a partial operation, defined if $m_{1}=m_{2}=m$, and is given by
$$u \obar v=
\begin{matrix}
u_{1,1} & \cdots & u_{1,n_{1}} & v_{1,1} & \cdots & v_{1,n_{2}}  \\
\vdots & & \vdots & \vdots & & \vdots \\
u_{m,1} & \cdots & u_{m,n_{1}} & v_{m,1} & \cdots & v_{m,n_{2}}
\end{matrix}$$

\item The row concatenation of $u$ and $v$, denoted by $\ominus$, is a partial operation, defined if $n_{1}=n_{2}=n$, and  is given by 


$$u \ominus v=
\begin{matrix}
u_{1,1} & \cdots & u_{1,n}  \\
\vdots & & \vdots \\
u_{m_{1},1} & \cdots & u_{m_{1},n} \\
v_{1,1} & \cdots & v_{1,n} \\
\vdots & & \vdots \\
v_{m_{2},1} & \cdots & v_{m_{2},n}
\end{matrix}$$

\end{enumerate} 

\end{Definition}

Note that, the operations of column and row concatenations are associative but not commutative and $\Lambda$ is the neutral element for both the operations.

With these operations defined, the definitions of a subword, prefix, primitive word and palindromes follow.

\begin{Definition} \cite{Kulkarni:2019}
Given $u \in \Sigma^{**}$, $v \in \Sigma^{**}$ is said to be a subword (respectively, proper subword) of $u$, denoted by $v \leq_{sw} u$ (respectively $v <_{sw} u)$ if $u=x \obar (x' \ominus v \ominus y') \obar y$ or $u=x \ominus (x' \obar v \obar y') \ominus y$ for some $x,x',y,y' \in \Sigma^{**}$ (respectively if any of $x,x',y,y'$ are non-empty).
\end{Definition}

\begin{Definition}\cite{Kulkarni:2019}
Let $u \in \Sigma^{**}$. An array $v \in \Sigma^{**}$ is said to be a prefix of $u$ $($suffix of $u$, respectively$)$, denoted by $v \leq_{p}^{2d} u$  $(v \leq_{s}^{2d} u$, respectively$)$ if $u=(v \ominus x) \obar y$  $(u=y \obar (x \ominus v)$, respectively$)$ for some $x,y \in \Sigma^{**}$. Furthermore, $v$ is said to be a proper prefix of $u$ $($proper suffix of $u$, respectively$)$ denoted by $v <_{p}^{2d} u$ $(v <_{s}^{2d} u$, respectively$)$ if either $x \neq \Lambda$ or $y \neq \Lambda$, or both $x,y \in \Sigma^{++}$.


\end{Definition}

\begin{Definition}\cite{Gamard:2017}
If $x \in \Sigma^{++}$, then by $(x^{k_{1} \obar})^{k_{2} \ominus}$ we mean that the array is constructed by repeating $x$, $k_{1}$ times column-wise to get $x^{k_{1} \obar}$, and repeating $x^{k_{1} \obar}$, $k_{2}$ times row-wise. An array $w \in \Sigma^{++}$ is said to be 2D \textit{primitive} if $w=(x^{k_{1} \obar})^{k_{2} \ominus}$ implies that $k_{1}k_{2}=1$ and $w=x$. 
\end{Definition}

By $Q_{2d}$, let us denote the set of all 2D primitive arrays. Also, if $w=(x^{k_{1} \obar})^{k_{2} \ominus}$ and $x$ is 2D primitive, then $x$ is said to be a 2D-\textit{primitive root} of $w$ denoted by $\rho_{2d}(w)$. Note that 2D-\textit{primitive root} is always unique for a given array.

\begin{Definition}\cite{Berthe:2001}
Let $u=[u_{i,j}]_{1 \leq i \leq m,1 \leq j \leq n}$ be an array of size $(m,n)$. The reverse image of $u$, denoted by $u^{R}$ is $[u_{m-i+1,n-j+1}]_{1 \leq i \leq m,1 \leq j \leq n}$.
Furthermore, if $u$ is equal to its reverse image $u^{R}$, then $u$ is said to be a \textit{two-dimensional palindrome}. 
By $P_{2d}$, we denote the set of all 2D palindromes in $\Sigma^{**}$.
\end{Definition}

Just for completion, recall that the transpose of $u=[u_{i,j}]_{1 \leq i \leq m,1 \leq j \leq n}$, denoted by $u^{T}$ is defined as: 
\begin{align*}
    u^{T}=(u_{1,1} \obar u_{2,1} \obar \cdots \obar u_{m,1}) \ominus \cdots \ominus (u_{1,n} \obar u_{2,n} \obar \cdots \obar u_{m,n}).
\end{align*}

\subsection{Two-dimensional Fibonacci Words}\label{sect3}

We are familiar with the Fibonacci numerical sequence $F(n)$ is defined recursively as $F(0)=1$, $F(1)=1$, $F(n)=F(n-1)+F(n-2)$ for $n \geq 2$. Similarly, the sequence $\{f_{n}\}_{n \geq 0}$ of Fibonacci words over $\Sigma=\{a,b\}$, is defined recursively by $f_{0}=a$, $f_{1}=b$, $f_{n}=f_{n-1}f_{n-2}$ for $n \geq 2$ . Note that the Fibonacci words are always defined over the binary alphabet and $|f_{n}|=F(n)$ for $n \geq 0$.

The 2D extension to Fibonacci arrays is defined in \cite{Apostolico:2000}, as below. 

\begin{Definition}
\cite{Apostolico:2000} Let $\Sigma=\{a,b,c,d\}$. The sequence of Fibonacci arrays, $\{f_{m,n}\}_{m,n \geq 0}$, is defined as:
\begin{enumerate}
\item $f_{0,0}=\beta,f_{0,1}=\gamma,f_{1,0}=\delta,f_{1,1}=\alpha$  where $\alpha, \beta,\gamma$ and $\delta$ are symbols from $\Sigma$ with some but not all, among $\alpha,\beta,\gamma$ and $\delta$ might be identical.
\item For $k \geq 0$ and $m,n \geq 1$, $$f_{k,n+1}=f_{k,n} \obar f_{k,n-1}, \hspace{0.2cm}  f_{m+1,k}=f_{m,k} \ominus f_{m-1,k}.$$
\end{enumerate} 
\end{Definition}

For convenience, throughout this paper we fix $f_{0,0}=a,f_{0,1}=b,f_{1,0}=c,f_{1,1}=d$, where some but not all of $a,b,c$ and $d$ might be identical. Let us call $f_{k,n+1}=f_{k,n} \obar f_{k,n-1}$ as column-wise expansion and $f_{m+1,k}=f_{m,k} \ominus f_{m-1,k}$ as row-wise expansion.  Example \ref{ex_for_fmn_defn} explains the construction of $f_{2,3}$ in two ways. In the first way, row-wise expansions precede column-wise expansions and in the second way  column-wise expansions precede row-wise expansions. 

\begin{Example}\label{ex_for_fmn_defn}
Let $\Sigma=\{a,b,c,d\}$. Then,
\begin{align*}
    f_{2,3}=f_{1,3} \ominus f_{0,3}&=(f_{1,2} \obar f_{1,1}) \ominus (f_{0,2} \obar f_{0,1})\\
    &=(f_{1,1} \obar f_{1,0} \obar f_{1,1}) \ominus (f_{0,1} \obar f_{0,0} \obar f_{0,1}).
\end{align*}
Or,
\begin{align*}
    f_{2,3}=f_{2,2} \obar f_{2,1}&=f_{2,1} \obar f_{2,0} \obar f_{2,1}\\
    &=(f_{1,1} \ominus f_{0,1}) \obar (f_{1,0} \ominus f_{0,0}) \obar (f_{1,1} \ominus f_{0,1}).
\end{align*}
Since, $f_{0,0}=a,f_{0,1}=b,f_{1,0}=c,f_{1,1}=d$, $f_{2,3}$ we can write, 
\begin{align*}
    f_{2,3}=\begin{matrix}
d & c & d\\
b & a & b
\end{matrix}
\end{align*}
\end{Example}

We state here some required properties of $f_{m,n}$.

\begin{Lemma}\cite{Kulkarni:2019}\label{lemprop}
Let $f_{m,n} ,  (m,n = 0,1,2,\dotsc)$ be the sequence of 2D Fibonacci arrays over  $\Sigma = \{ a,b,c,d\}$, with $f_{0,0} = a, f_{0,1} = b, f_{1,0} = c, f_{1,1} = d$. Also let $\Sigma_1 = \{a,b\}$, $\Sigma_2 = \{c,d\}$, $\Sigma_1' = \{a,c\}$ and $\Sigma_2' = \{b,d\}$ such that $\Sigma = \Sigma_1 \cup \Sigma_2 = \Sigma_1' \cup \Sigma_2'$. Then,  
\begin{itemize}
\item[a.] Any row of $f_{m,n}$ is a 1D Fibonacci word over either $\Sigma_1$ or $\Sigma_2$.

\item[b.] If $\Sigma_1 \ne \Sigma_2$ then all the rows of $f_{m,n}$, over $\Sigma_1$ are identical and all the rows of $f_{m,n}$, over $\Sigma_2$ are identical.

\item[c.] Any column of $f_{m,n}$ is a 1D Fibonacci word over either $\Sigma_1'$ or $\Sigma_2'$.

\item[d.] If $\Sigma_1' \ne \Sigma_2'$ then all the columns of $f_{m,n}$, over $\Sigma_1'$ are identical and all the columns of $f_{m,n}$, over $\Sigma_2'$ are identical. 

\item[e.] If $\Sigma_1 = \Sigma_2 (\Sigma_1'=\Sigma_2')$, then either all the rows$($columns$)$ of $f_{m,n}$ are identical or a set of rows are identical and are complementary to the set of remaining rows $($columns, respectively$)$ which are identical.
\end{itemize}
\end{Lemma}

For more properties of 2D Fibonacci words the reader can  refer \cite{Kulkarni:2019} and \cite{Mahalingam:2018}.

\section{Tandem Repeats in \texorpdfstring{$f_{m,n}$}{}}
In molecular genetics, data is represented as a sequence of characters. Many algorithms in molecular biology try to find structures called tandems, which are patterns of nucleotides repeated adjacent to each other in a DNA. In fact more than half of the human genome contains repeated sequences \cite{Lander:2001,Mundy:2004}. They are related to inherited traits of an individual and play a role in DNA tests. Even an approximate tandem repeat is related to some human diseases \cite{Landau:1993}. Tandem repeats in $2D$ words were studied in \cite{Apostolico:2000, Charalampopoulos:2020, Amir:2020a, Amir:2020b} with varying interests.  

\subsection{Types of Tandems}
\begin{Definition} \cite{Apostolico:2000}
In a two dimensional array $X$, a tandem is a configuration consisting of two occurrences of a same primitive block $W$ that touch each other with one side($Type~I$ tandem) or with a corner($Type~II$ tandem). 
\end{Definition}

We further divide Type I and Type II tandems as defined below.
\begin{Definition}
Let $u$ be a 2D word. For a primitive block $W$, the $Type~I$ tandem $W \obar W$ is called a $Type~I(a)$ tandem of $u$ and  $W \ominus W$ is called a $Type~I(b)$ tandem of $u$.
\end{Definition}

\begin{Definition}
Let $u$ be a 2D word. For a primitive block $W$ which is a subarray of $u$, and for $A$,$B \in \Sigma^{++}$, whose sizes are the same as that of $W$, if the $Type~II$ tandem is such that $(W \obar A) \ominus (B \obar W)$ is a sub array of $u$, then it is called a $Type~II(a)$ tandem of $u$. And if the $Type~II$ tandem is such that $(C \obar W) \ominus (W \obar D)$ is a sub array of $u$, for some $C$,$D \in \Sigma^{++}$, whose sizes are the same as that of $W$, then it is called a $Type~II(b)$ tandem of $u$.
\end{Definition}

All the four types of tandems are shown in Figure \ref{4types}.
\begin{figure}[ht]
    \centering
    \includegraphics[scale=0.35]{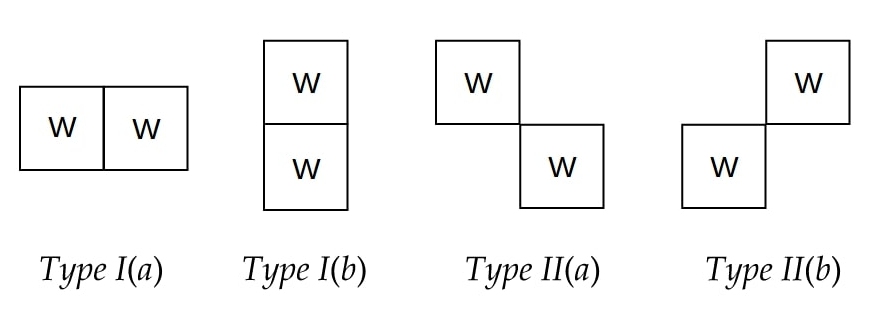}
    \caption{Types of tandems that can occur in a two-dimensional array}
    \label{4types}
\end{figure}

\begin{Remarks*} 
Block $W$ is called the root of the tandem. A tandem need not be a 2D array. Also note that, $Type~I(b)~(Type~II(b)$, respectively$)$ tandem will be a $90^\circ$ rotation of $Type~I(a)~(Type~II(a)$, respectively$)$ and vice versa.
\end{Remarks*}

Formulas to count the exact number of squares including repetitions and without including repetitions(i.e. distinct) are established in \cite{Fraenkel:1999}. Both the theorems are recalled here. Note that, there is a corrigendum \cite{Fraenkel:2014} to \cite{Fraenkel:1999} correcting an error in the formula for number of squares, derived in \cite{Fraenkel:1999}. The correct formula was verified in \cite{Du:2016} also, using a logic-based decision procedure, implemented as $Walnut$.

Denote by $D(n)$ and by $R(n)$ the exact number of distinct and repeated squares, respectively, in $f_n$. Then,

\begin{Theorem}\cite{Fraenkel:1999}\label{Dofn}
For $n \ge 5$, $D(n) = 2(F(n-2) -1)$.
\end{Theorem}

\begin{Theorem}\cite{Fraenkel:2014}\label{Rofn}
For $n \ge 3$, $R(n) = \frac{4}{5}nF(n)-  \frac{2}{5}(n+6)F(n-1) - 4F(n-2)+n+1$.
\end{Theorem}
\section{Number of Tandems of \texorpdfstring{$f_{m,n}$}{}}\label{RepeatTand}
In this section we count the  number of tandems (with repetition included, i.e. same tandems, but positioned at different locations are included in the counting) in the given Fibonacci array $f_{m,n}$. 

\subsection{The number of \texorpdfstring{$Type~I$}{} Tandems in the Fibonacci Array \texorpdfstring{$f_{m,n}$}{}}

First we prove two Propositions to help our counting.
\begin{Proposition}\label{count1}
Let $S$ be a $Type~I(a)$ tandem of size $(r,c)$, with $r \ge 1$ and $c \ge 2$ is even. Then there are $\frac{r(r+1)}{2}$ number of $Type I(a)$ tandems in $S$. 
\end{Proposition}
\begin{proof}
Observe that if $S$ is a $Type~ I(a)$ tandem, then every row of $S$ will be a $Type ~I(a)$ tandem. Conversely, if every row of an array is a $Type ~I(a)$ tandem then the array itself will be a $Type ~I(a)$ tandem. 

Denote the top left corner of $S$ as position $[1,1]$. Consider the $Type~ I(a)$ tandem, $S_1$, with $r$ rows, and $c$ columns, located at $[1,1]$ (i.e. $S_1 = S$ itself). The sub arrays of $S_1$ consisting of its first $i$ rows alone, $1 \le i \le r$, will be $Type ~I(a)$ tandems of sizes $(i,c)$. Hence there are $r$ tandems, located at $[1,1]$. Now consider the tandem, $S_2$, with $(r-1)$ rows, located at $[2,1]$. The sub arrays of $S_2$ consisting of its first $i$ rows alone, $1 \le i \le r-1$ will be $Type ~I(a)$ tandems of sizes $(i,c)$. Hence there are $r-1$ tandems, located at $[2,1]$. Continuing this process we get $r+(r-1)+ \cdots + 1 = \frac{r(r+1)}{2}$ $Type~ I(a)$ tandems in $S$.
\end{proof}

\begin{Proposition}
Let $S$ be a $Type~I(b)$ tandem of size $(r,c)$, with $c \ge 1$ and $r \ge 2$ is even. Then there are $\frac{c(c+1)}{2}$ number of $Type ~I(b)$ tandems in $S$. 
\end{Proposition}
\begin{proof}
The proof is similar to the proof of Proposition\ref{count1}.
\end{proof}

Now through a combinatorial argument we can count the number of $Type~ I (a)$ and $Type~I (b)$ tandems in $f_{m,n}$.
\begin{Theorem}\label{t1a}
For $m \ge 1$ and $n \ge 3$, let $R(m,n;I(a))$ denote the exact number of $Type~I(a)$ tandems occurring in $f_{m,n}$. Then,
\begin{center}
$R(m,n;I(a)) = R(n) \frac{F(m) (F(m)+1)}{2}$,
\end{center}
 where $R(n)$ is the number of squares in the Fibonacci word $f_n$ (see Theorem \ref{Rofn}).
\end{Theorem}
\begin{proof}
Due to the characteristics of $f_{m,n}$, as mentioned in the Lemma \ref{lemprop}, if there is a $Type~ I(a)$ tandem (square) of length $2l$, $1\le l \le \lfloor \frac{F(n
)}{2} \rfloor $ in the first row, occurring at the position $[1,s], 1\le s \le F(n)-1$, then there will be squares of the same length $2l$, at the positions $[i,s]$ for all $2 \le i \le F(m)$. Therefore, there will be a $Type~ I(a)$ tandem of size $(F(m),2l)$ located at $[1,s]$. Now by Proposition \ref{count1}, we get $\frac{F(m)(F(m)+1)}{2}$ number of $Type~ I(a)$ tandems from this tandem. Note that, for every square present in the first row of $f_{m,n}$, we get $\frac{F(m)(F(m)+1)}{2}$ number of $Type ~I(a)$ tandems.

By Theorem \ref{Rofn}, we have,
\[
R(n) = \frac{4}{5} nF_n - \frac{2}{5}(n+6)F_{n-1} -4F_{n-2}+n+1~
\] number of squares present in the first row of $f_{m,n}$. Hence, $f_{m,n}$ has $R(n) \frac{F(m) (F(m)+1)}{2}$ number of $Type~ I(a)$ tandems. 
\end{proof}

\begin{Theorem}
For $m \ge 3$ and $n \ge 1$, Let $R(m,n;I(b))$ denote the exact number of $Type~ I(b)$ tandems occurring in $f_{m,n}$. Then,
\begin{center}
$R(m,n;I(b)) = R(m) \frac{F(n) (F(n)+1)}{2}$,
\end{center}
 where $R(m)$ is the number of squares in the Fibonacci word $f_m$ (see Theorem \ref{Rofn})
\end{Theorem}
\begin{proof}
The proof is similar to the proof of Theorem \ref{t1a}.
\end{proof}

\subsection{The Number \texorpdfstring{$Type~ II$}{} Tandems in the Fibonacci Array \texorpdfstring{$f_{m,n}$}{}}
In this section we count the  number of $Type~ II$ tandems (repetition included) in the given Fibonacci array $f_{m,n}$.

\begin{Proposition} \label{pro5}
Let $\Sigma = \{a,b,c,d\}$ be the alphabet. Let $A,B,C,D \in \Sigma^{++}$ be arrays of same size such that $(A \obar B) \ominus (C \obar D)$ is a sub array  of $f_{m,n}$. Then, $A = D = W$ (i.e. a $Type~II(a)$ tandem occurs) iff $A = B = C = D = W$.     
\end{Proposition}
\begin{proof}
If $A = B = C = D = W$, then obviously there occurs a $Type ~II(a)$ tandem. Conversely, Let $A = D = W$ (i.e. a $Type ~II(a)$ tandem occurs) in $(A \obar B) \ominus (C \obar D)$. Then $A$ and $D$ have identical first rows over a same alphabet say $\Sigma' \subset \Sigma$. By Lemma \ref{lemprop}, the first rows of $(A \obar B)$ and $(C \obar D)$ will be identical. A similar argument on the other rows show that $A = B = C = D = W$.     
\end{proof}
\begin{Definition}\cite{Apostolico:2000} \label{quarticdefn}
Given a 2D array $X$, a quartic in $X$ is a configuration consisting of the form 
\begin{tabular}{|c|c|}
\hline
$W$&$W$\\
\hline
$W$&$W$\\
\hline
\end{tabular}
where block $W$ is primitive.
\end{Definition}
From Proposition \ref{pro5} and Definition \ref{quarticdefn}  we infer that, in Fibonacci arrays, $Type ~II(a)$ tandems occur only as a part of quartics. Hence, the number of $Type~ II(a)$ tandems in a Fibonacci array equals the number of quartics in the Fibonacci array. In the next theorem we count the number of quartics and hence the number of $Type~ II(a)$ tandems.
\begin{Theorem}\label{t2a}
For $m,n \ge 3$, let $R(m),R(n)$ denote the number of squares $($repetitions included$)$ in the Fibonacci words $f_m,f_n$ respectively. Then, there are $R(m)R(n)$ quartics in the Fibonacci array $f_{m,n}$ and hence there are as many $Type~ II(a)$ tandems in $f_{m,n}$.
\end{Theorem}   
\begin{proof}
We know that, every row of a $Type ~I(a)$ tandem will be a $Type ~I(a)$ tandem and every column of $Type ~I(b)$ tandem will be a $Type~I(b)$ tandem. Since a quartic is both a $Type ~I(a)$ and a $Type ~I(b)$ tandem, only the sub arrays having the intersection of a $Type ~I(a)$ tandem  and a $Type ~I(b)$ tandem as its domain, can be quartics of $f_{m,n}$. In fact a $Type ~I(a)$ tandem of size $(F(m),2l)$, $1\le l \le \lfloor \frac{F(n
)}{2} \rfloor$ and  a $Type ~I(b)$ tandem of size $(2k,F(n))$, $1\le k \le \lfloor \frac{F(m
)}{2} \rfloor$ will create a quartic of size $(2k,2l)$ (Refer Figure \ref{quarfig}).  Now, as there are $R(m),R(n)$ number squares, respectively, in the Fibonacci words $f_m,f_n$ there will be $R(m)$ number of squares in the first column and $R(n)$ number of squares the first row of $f_{m,n}$. Hence, there will be $R(m)R(n)$ quartics in the Fibonacci array $f_{m,n}$. Therefore by proposition \ref{pro5} there are as many $Type ~II(a)$ tandems in $f_{m,n}$. (Values of $R(m)$ and $R(n)$ are given in proposition \ref{count1} as derived in \cite{Fraenkel:1999}) .
\end{proof}

\begin{figure}[ht]
    \centering
    \includegraphics[scale=.35]{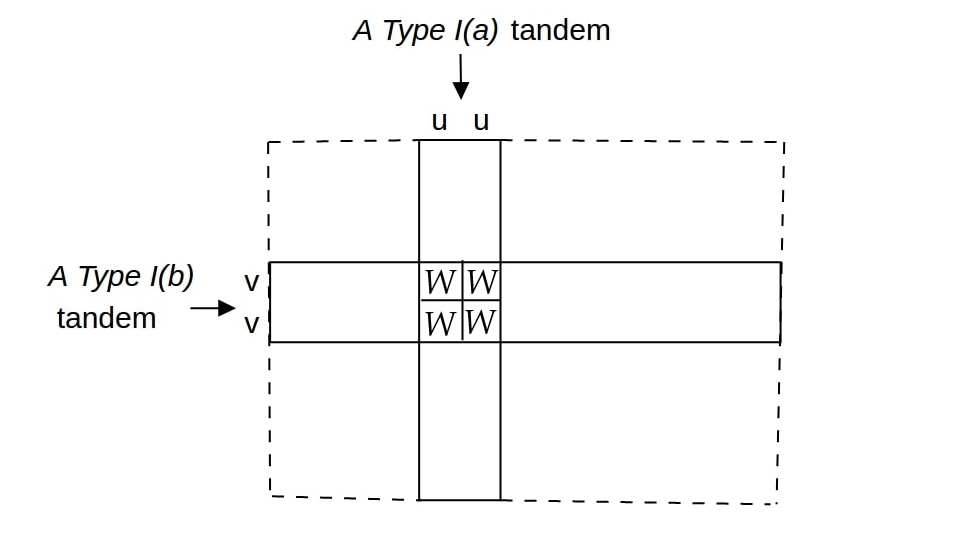}
    \caption{Formation of a quartic}
    \label{quarfig}
\end{figure}

\begin{Theorem}\label{t2b}
There are $R(m)R(n)$ number of $Type~ II(b)$ tandems in $f_{m,n}$.
\end{Theorem}
\begin{proof}
Like $Type~ II(a)$ tandems, a $Type~ II(b)$ tandem can occur as a part of some quartic only. Hence the number of $Type~ II(b)$ tandems will be the same as the number of quartics in $f_{m,n}$. Hence by the Theorem \ref{t2a} there will be $R(m)R(n)$ number of $Type~ II(b)$ tandems in $f_{m,n}$.
\end{proof}

\section{Number of Distinct Factors of \texorpdfstring{$f_{n}$}{}}
We derive an important result of factor complexity of finite 1D Fibonacci words here. This will help us to count the distinct $Type~I$, $Type~II$ tandems in $f_{n}$ and $f_{m,n}$. Analysing the factor complexity of an infinite word is crucial as we can conclude many things about the nature of the word. For an infinite word $u$, the factor complexity function is bounded if $u$ is ultimately periodic. Conversely, $u$ will be ultimately periodic, if $p_n(u) \le n$ for some $n$ \cite{Shallit:2003}. Balanced words, Sturmian words, Arnoux-Rauzy word over a ternary alphabet all have been characterised through their complexity functions. Entropy of an infinite word and unavoidable patterns in an infinite word are also closely related to the concept complexity function.

\begin{Definition}\label{subwodcomplexity}\cite{Lothaire:2002}
For an infinite word $w$, the subword complexity function of $w$, $p_w(n)$, counts the number of distinct subwords of length $n$ in $w$. The subword complexity sequence of $w$ is the sequence $p_w = (p_w(1), p_w(2), p_w(3), \dotsc)$.  
\end{Definition}

Recall that, the 1D infinite Fibonacci word \cite{Lothaire:2002},
$$f_{\infty}= \lim_{n\to\infty} h^n(a) =h^{\omega}(a)= babbabab\cdots \cdots$$
is the fixed point of the morphism $h(a)=b, h(b)=ba$ on $\Sigma = \{a,b\}$.
\begin{Proposition}\label{p_of_w}\cite{Lothaire:2002}
 The subword complexity function of $f_{\infty}$ is $p_{f_{\infty}}(n) = n+1$. That is the infinite Fibonacci word has exactly $n+1$ factors of length $n$.  
\end{Proposition}

Proposition \ref{p_of_w} says that the 1D infinite Fibonacci word $f_{\infty}$ has exactly $n+1$ factors of length $n$ for any $n \ge 1$. This characterisation itself is sometimes used as the definition of 1D Fibonacci words (in fact the broad category of 1D Sturmian words). Though the count for distinct factors $($of any length$)$ of the infinite Fibonacci word is already available, the count for the number of distinct factors of the finite Fibonacci word $f_n$, $n\ge 2$, is not available. In this section we derive a formula for this count. Later we use it to count the number of distinct tandems in a given 2D Fibonacci array, $f_{m,n}$.

We recall the following theorem regarding the positions of all the distinct factors of length $k \ge 1$, of $f_{\infty} = f$.
\begin{Theorem}\label{all fact} \cite{Chuan:2005} Let $n \ge 2$ and $F(n) \le k < F(n+1)$. Define 
$$z_{j}^{(k)}= \begin{cases}
f[j+1;k], & 0 \le j \le F(n)-1\\
f[j+F(n+1)-k;k], & F(n) \le j \le k
\end{cases}$$ where $f[i';j']$, $j' \ge i'$ is the sub word of $f$ of length $j'$ starting at position $i'$. Then the words $z_{0}^{(k)}$,$z_{1}^{(k)}$,$\cdots$,$z_{k}^{(k)}$ are the $k+1$ distinct factors of $f$ of length $k$ which are listed in the order of their first occurrences in $f$.   
\end{Theorem}

We now state and prove one of the major theorems of this paper, related to the factor complexity of $f_{n}$.

\begin{Theorem}\label{factcomplex1d}
For $n \ge 2$, let $p_{k}(f_n)$ denote the number of distinct subwords of length $k$, $1 \le k \le F(n)$ in $f_n$. Then,
$$p_{k}(f_n)= \begin{cases}
k+1~, & 1 \le k \le F(n-2)\\
F(n-2)+2~, & F(n-2)+1 \le k \le F(n-1)-1\\
F(n)+1 -k~, & F(n-1) \le k \le F(n)
\end{cases}$$
 
\end{Theorem}
Given $f_n$, proof is done in three cases.\\

Table \ref{distfact} lists the number of distinct factors of various lengths that occur in $f_n$ for $n=2,3,4,5,6$. 

\begin{table}[ht]
\caption{Number of distinct factors of various lengths occurring in $f_n$}
\centering
\begin{tabular}{c|ccccccccccccc}
  Factor Length $\rightarrow$ & 1 & 2 & 3 & 4 & 5 & 6 & 7 & 8 & 9 & 10 & 11 & 12 & 13   \\
     $f_n$ and $F(n)$ $\downarrow$& & & & & & & & & & & & &   \\
     \hline
     $f_2$ , $F(2) = 2$ & 2 & 1 &  &  &  &  &  &  &  &  &  &  &     \\
     $f_3$ , $F(3) = 3$ & 2 & 2 & 1 &  &  &  &  &  &  &  &  &    \\
     $f_4$ , $F(4) = 5$ & 2 & 3 & 3& 2 & 1 &  &  &  &  &  &  &  &    \\
     $f_5$ , $F(5) = 8$ & 2 & 3 & 4 & 5 & 4 & 3 & 2 & 1 &  &  &  &  &    \\
     $f_6$ , $F(6) = 13$ & 2 & 3 & 4 & 5 & 6 & 7 & 7 & 6 & 5 & 4 & 3 & 2 & 1   \\
\end{tabular}
\label{distfact}
\end{table}
\begin{Corollary}\label{cor1}
For $n \ge 2$, let $p(f_n)$ denote the number of distinct subwords of $f_n$. Then $p(f_n) = \sum_{k=1}^{F(n)} p_{k}(f_n)$.
\end{Corollary}
\begin{proof}
As $p_{k}(f_n)$ denote the number of distinct subwords of length $k$ in $f_n$, summation over $k$ yields the value of $p(f_n)$.
\end{proof}

\section{Number of Distinct Tandems of \texorpdfstring{$f_{m,n}$}{}}
We use the result obtained in Theorem \ref{factcomplex1d} to count the number of distinct tandems in $f_{m,n}$. 

Recall from Theorem \ref{Dofn} that, For $n \ge 5$, $D(n) = 2(F(n-2)-1)$, where $D(n)$ denotes the exact number of distinct squares in $f_n$. 

\subsection{The Number of Distinct \texorpdfstring{$Type ~I$}{} Tandems in \texorpdfstring{$f_{m,n}$}{}}

\begin{Theorem}\label{2dDn}
For $m \ge 2$ and $n \ge 5$, let $D(m,n;I(a))$ denote the exact number of distinct $Type~ I(a)$ tandems occurring in $f_{m,n}$. Then,
\begin{center}
$D(m,n;I(a)) = D(n)~p(f_m)$.
\end{center} 
\end{Theorem}  
\begin{proof}
The counting process is very similar to the process explained in the proof of Theorem \ref{t1a}, except that, we have to consider only the distinct squares in the first row and the distinct factors of the first column. Hence by Theorem \ref{Dofn}(number of distinct squares in any row) and Corollary \ref{cor1}(number of distinct factors in $f_m$), the result follows. 
\end{proof}

\begin{Theorem}
For $m \ge 5$ and $n \ge 2$, let $D(m,n;I(b))$ denote the exact number of distinct $Type ~I(b)$ tandems occurring in $f_{m,n}$. Then,
\begin{center}
$D(m,n;I(b)) = D(m)~p(f_n)$.
\end{center} 
\end{Theorem}
\begin{proof}
The proof is similar to the proof of Theorem \ref{2dDn} with the roles of $m$ and $n$ interchanged.
\end{proof}
\subsection{The Number of Distinct \texorpdfstring{$Type~ II$}{} Tandems in \texorpdfstring{$f_{m,n}$}{}}
In this section we count the number of distinct $Type~ II$ tandems in $f_{m,n}$. 

\begin{Theorem}
For $m,n \ge 3$, let $D(m),D(n)$ denote the number distinct squares in the Fibonacci words $f_m,f_n$ respectively. Then, there are $D(m)D(n)$ distinct quartics in the Fibonacci array $f_{m,n}$; and hence there are as many distinct $Type~ II(a)$ and $Type~ II(b)$ tandems in $f_{m,n}$.
\end{Theorem}  
\begin{proof}
The proof is similar to Theorem \ref{t2a} except that only distinct squares are considered in the counting.
\end{proof}

\section{The 2D Infinite Fibonacci Word, \texorpdfstring{$f_{\infty,\infty}$}{}}

We know that, a morphism is a map $h: \Sigma^*\rightarrow\Delta^*$, where $\Sigma$,$\Delta$ are alphabets, such that $h(xy) = h(x) h(y)$ for all strings $x,y \in \Sigma^*$. If $\Sigma = \Delta$ we can iterate $h$. That is $h^n(x) =h(h^{n-1}(x)), n\ge 2 $. In the literature, analysis of the Fibonacci language, $\{f_{n}\}_{n \geq 0}$, over the alphabet $\Sigma = \{a,b\}$, is done with the help of the Fibonacci morphism, $h: a \rightarrow b,b \rightarrow ba$. That is to say that, the infinite Fibonacci word $\mathbf{f_\infty} = babbababbabba \cdots$ is the limit of the sequence $\{f_{n}\}_{n \geq 0}$. That is, $$f_{\infty}= \lim_{n\to\infty} h^n(a) =h^{\omega}(a)= babbabab\cdots \cdots$$
is the fixed point of the morphism $h(a)=b, h(b)=ba$ on $\Sigma = \{a,b\}$.

Also, note that, finitely many iterations of the above morphism ($h(a)=b,h(b)=ba$) generates the intermediate finite Fibonacci words, $f_{n}$ for all $n \geq 0$, ultimately leading to $f_{\infty}$, the fixed point of the morphism, when $n \rightarrow \infty$. 
\subsection{\texorpdfstring{$d-$}{}dimensional Morphisms}
Similar to $f_{\infty}$ we define $f_{\infty,\infty}$, the 2D infinite Fibonacci word. It is noted that, the concept of two dimensional iterated morphisms was first introduced in \cite{Arnoux:2004}. But, such morphisms might not always result in  rectangular patterns. Interestingly, in \cite{Charlier:2010}, $d-$dimensional morphisms are introduced and using them $d-$dimensional infinite words are defined. After stating the required definitions and results from \cite{Charlier:2010, Lecomte:2000}, we will define and validate the 2D morphism generating $f_{\infty,\infty}$. As a continuation, in section \ref{DFAOsec} , we construct a DFAO (Deterministic Finite Automaton with Output) which generates $f_{\infty,\infty}$.

\subsection{The 2D Fibonacci Morphism}
Now, we state and prove another important theorem of this paper. Consider the case $d=2$ in \cite{Charlier:2010}.
\begin{Definition}
Let $\mu: \Sigma \rightarrow B_2(\Sigma)$ be a map and $x$ be a 2-dimensional array such that,
\begin{equation}
\label{condn2dim}
    \forall i \in \llbracket 1,2 \rrbracket, \forall k < |x|_i, \forall a,b \in Fact_\textbf{1}(x_{|i,k}) : |\mu(a)|_i=|\mu(b)|_i.
\end{equation} 
Then the image of x by $\mu$ is the 2-dimensional array defined as,
$$\mu(x) = \odot_{0\le n_1 < |x|_1}^{1} \left ( \odot_{0\le n_2 < |x|_2}^{2} \mu(x(n_1,n_2)) \right ). $$
\end{Definition}

If for all $a \in \Sigma$ and all $n \ge 1$, $\mu^{n-1}(a)$ satisfies (\ref{condn2dim}), then $\mu$ is said to be a 2-dimensional morphism. Further, the properties of $\mu$ being prolongable on $a$ and $\mu$ having a fixed point, follow automatically.

Now, consider the map,
\begin{equation}
\label{FibMap}
\mu : ~~d \rightarrow \begin{matrix}
d & c\\
b & a
\end{matrix},~~c \rightarrow \begin{matrix}
d \\
b 
\end{matrix},~~b\rightarrow \begin{matrix}
d & c
\end{matrix},~~~a  \rightarrow 
d.    
\end{equation}
We have written in the order $d,c,b,a$ as $(\mu(d))_\textbf{0} = d$ (i.e. as $\mu$ is prolongable on $d$).

The following result is proved by induction technique.

\begin{Theorem}
Let $f_{0,0} = a,~f_{0,1} = b,~f_{1,0} = c,~f_{1,1} = d$ and $\mu$ be defined as in (\ref{FibMap}). Then for $m,n \ge 1,$ $$\mu(f_{m,n}) = f_{m+1,n+1}.$$ 
\end{Theorem}

\begin{Corollary}
For $n \ge 1,~~\mu(f_{n,n}) = f_{n+1,n+1}.$ 
\end{Corollary}
\begin{proof}
By taking $m=n$ in the theorem, the corollary follows.
\end{proof}

\begin{Corollary}
$w = \mu^{\omega}(d)$ exists. 
\end{Corollary}
\begin{proof}
Since, $\mu(f_{n,n}) = f_{n+1,n+1}$, $\mu^n(d)$ is inductively well defined from $\mu^{n-1}(d)$ and hence $\mu$ is a 2D morphism. Now, since $(\mu(d))_\textbf{0} = d$, $\mu$ is prolongable on $d$, and $$w = \lim_{n \rightarrow +\infty} \mu^n(d) = \mu^{\omega}(d)$$ exists. This fixed point $w$ is called the infinite 2D Fibonacci word and is denoted by $f_{\infty,\infty}$.
\end{proof}

First few iterations of $\mu$ on $d$ are shown below.
$$d \rightarrow \begin{matrix}
d & c\\
b & a
\end{matrix} \rightarrow \begin{matrix}
d & c & d\\
b & a & b\\
d & c & d
\end{matrix} \rightarrow \begin{matrix}
d & c & d & d & c \\
b & a & b & b & a \\
d & c & d & d & c \\
d & c & d & d & c \\
b & a & b & b & a \\
\end{matrix} \rightarrow \begin{matrix}
d & c & d & d & c & d & c & d &\cdots\\
b & a & b & b & a & b & a & b & \cdots\\
d & c & d & d & c & d & c & d &\cdots\\
d & c & d & d & c & d & c & d &\cdots\\
b & a & b & b & a & b & a & b &\cdots\\
d & c & d & d & c & d & c & d &\cdots\\
b & a & b & b & a & b & a & b & \cdots\\
d & c & d & d & c & d & c & d &\cdots\\
\vdots & \vdots & \vdots &\vdots & \vdots & \vdots & \vdots & \vdots & \ddots
\end{matrix}$$

\subsection{2D Finite Fibonacci Words as Morphic Words} 
 As $\mu(f_{m,n}) = f_{m+1,n+1}$, it is interesting to note that, the morphism $\mu$ behaves like a shift operator. Hence, iterated applications of $\mu$ on $f_{1,n'}$ or $f_{m',1}$  with appropriate $m',n'$ values, can generate any finite 2D Fibonacci word.
\begin{Corollary}
Let $m,n \ge 2$ and $m \ne n$. Then, we have 
$$f_{m,n} = 
\left\{
	\begin{array}{ll}
		\mu^{(m-1)}(f_{1,n-m+1})  & \mbox{if } m < n \\
		 \\
		\mu^{(n-1)}(f_{m-n+1,1})  & \mbox{if } m > n 
	\end{array}
\right.$$
\end{Corollary}

\begin{proof}

We know that the 2D Fibonacci words $f_{1,n'}$ and $f_{m',1}$ for any $m',n' \ge 2$ can be easily obtained, as they are literally 1D Fibonacci words over $\{d,c\}$ and $\{d,b\}$ respectively.

Now, since $\mu(f_{m,n}) = f_{m+1,n+1}$,~ for  $k \ge 1,~ \mu^k(f_{m,n}) = f_{m+k,n+k}$.

Therefore, 
$$	\begin{array}{ll}
		~~\mbox{if }~ m < n,~~~ & \mu^{(m-1)}(f_{1,n-m+1}) = f_{1+m-1,n-m+1+m-1} = f_{m,n}~~~~\mbox{and }\\
		 \\
		~~\mbox{if }~ m > n,~~~  &  \mu^{(n-1)}(f_{m-n+1,1}) = f_{m-n+1+n-1,1+n-1} = f_{m,n}. 
	\end{array}
$$
\end{proof}

\begin{Example}
Suppose we want to generate $f_{3,5}$. 

Since $n>m$~~, $n-m \ge 1$ and we start with $f_{1,n-m+1} = f_{1,3} = d~c~d$, so that, $\mu^{2}(f_{1,3})$ will be $f_{3,5}$.
\begin{align*}
  \mu^{2}(f_{1,3}) = \mu^{2}(d~c~d) &= \mu \left( \begin{matrix}
d & c & d & d & c\\
 b & a & b &  b & a 
 \end{matrix} \right)  \\
 &= \begin{matrix}
d & c & d & d & c & d & c & d \\
b & a & b & b & a & b & a & b \\
d & c & d & d & c & d & c & d 
\end{matrix} = f_{3,5}.   
\end{align*}

\end{Example}

\section{Factor Complexities of \texorpdfstring{$f_{\infty,\infty}$}{} and \texorpdfstring{$f_{m,n}$}{}}
In this section we find the number of distinct subwords of any given $f_{m,n}$ and also of $f_{\infty,\infty}$.  The result proved in Theorem \ref{factcomplex1d}, will be used to find the factor complexity of  $f_{m,n}$.   We denote by $p_{k,l}(u)$, the complexity function of the two dimensional word $u$. It is understood through the subscripts $k$ and $l$ that $u$ is a two dimensional word. 

\begin{Proposition}\label{2dfactorcountprop}
For $u$, a 2D word over $\Sigma$, let $p_{k,l}(u), ~k,l \ge 1$ denote the number of subwords (subarrays) of $u$ of size $(k,l)$. Then, for the two dimensional infinite Fibonacci word, $f_{\infty,\infty}$, $p_{k,l}(f_{\infty,\infty}) = (k+1)(l+1)$.
\end{Proposition}
\begin{proof}
The result is immediate through a simple combinatorial argument. By lemma \ref{lemprop}, every row (every column) of $f_{\infty,\infty}$ written as a 1D word is an 1D infinite Fibonacci word $f_{\infty}$ over any one of the two letter alphabets $\{d,c\},\{b,a\} (\{d,b\},\{c,a\})$. By proposition \ref{p_of_w}, there are $k+1$ distinct subwords of length $k$ in every column (we call them, vertical factors) and $l+1$  distinct subwords of length $l$ in every row (we call them, horizontal factors). Since there are only two distinct columns (one over $\{d,b\}$ and one over $\{c,a\}$) in $f_{\infty,\infty}$, there are $2(k+1)$ vertical factors of length $k$. Similarly, there are $2(l+1)$ horizontal factors of length $l$. By pairing the vertical(horizontal) factors located in the same row(column), we get $k+1$ ( l+1 ) pairs of vertical(horizontal) factors. Now, note that a 2D subword of $f_{\infty,\infty}$ is formed by the process similar to the one explained in the proof of Theorem 2 from \cite{Mahalingam:2018}. That is, a vertical factor will produce a 2D factor through a horizontal factor, if both have the same prefix of size $(1,1)$. Now, as one factor in a pair of vertical factors shares a common prefix of size $(1,1)$ with one factor in a pair of horizontal factors, there will be $(k+1)(l+1)$ distinct factors of size $(k,l)$ in $f_{\infty,\infty}$. 
\end{proof}

To understand the process, in Example \ref{2dfactorexample} we list the factors of size $(2,3)$ in $f_{\infty,\infty}$.
\begin{Example} \label{2dfactorexample}
Consider $f_{\infty,\infty}$. 

\vspace{0.5cm}
$\begin{tabular}{|c|| c| c| c| c|}
\hline
 Horizontal & & & & \\
 Factors of & dcd & cdd & ddc & cdc \\
  length $3$ $\rightarrow$  & (or) &  (or) & (or) & (or)  \\
 $ Vertical $ & bab  &  abb  &  bba  &  aba    \\
  Factors of  &   &   &   &       \\
   length 2 $\downarrow$ &   &   &   &       \\
  \hline
  \hline

$
 \begin{tabular}{c}
      d \\
      b 
 \end{tabular} (or)  \begin{tabular}{c}
      c\\
      a 
 \end{tabular} $ 
 &  $\begin{tabular}{|ccc|}

\hline
d&c&d \\
 b&a&b\\
\hline
\end{tabular}$ &   $\begin{tabular}{|ccc|}

\hline
 c&d&d \\
 a&b&b\\
\hline
\end{tabular}$ &  $\begin{tabular}{|ccc|}

\hline
 d&d&c \\
 b&b&a\\
\hline
\end{tabular}$ &  $\begin{tabular}{|ccc|}

\hline
 c&d&c \\
 a&b&a\\
\hline
\end{tabular}$ \\

\hline

$
 \begin{tabular}{c}
      b \\
      d 
 \end{tabular}  (or) \begin{tabular}{c}
      a\\
      c 
 \end{tabular}$ &  $\begin{tabular}{|ccc|}

\hline
b&a&b\\
d&c&d \\
 
\hline
\end{tabular}$ &   $\begin{tabular}{|ccc|}

\hline
 a&b&b\\
 c&d&d \\

\hline
\end{tabular}$ &  $\begin{tabular}{|ccc|}

\hline
b&b&a\\
d&d&c \\
 
\hline
\end{tabular}$ &  $\begin{tabular}{|ccc|}

\hline
 a&b&a\\
 c&d&c \\

\hline
\end{tabular}$ \\
 \hline

$
 \begin{tabular}{c}
      d \\
      d 
 \end{tabular} (or) \begin{tabular}{c}
      c\\
      c 
 \end{tabular}$ &  $\begin{tabular}{|ccc|}

\hline
d&c&d \\
d&c&d \\
 
\hline
\end{tabular}$ &   $\begin{tabular}{|ccc|}

\hline
c&d&d \\
 c&d&d \\

\hline
\end{tabular}$ &  $\begin{tabular}{|ccc|}

\hline
d&d&c \\
d&d&c \\
 
\hline
\end{tabular}$ &  $\begin{tabular}{|ccc|}

\hline
c&d&c \\
 c&d&c \\

\hline
\end{tabular}$ \\
 \hline
 \end{tabular}$
 
\vspace{0.5cm}

We have $3 \times 4 = 12$ factors of size $(2,3)$.
\end{Example}

The count, carried out in proposition \ref{2dfactorcountprop} can be restricted to any finite Fibonacci word $f_{m,n}$. Note that, while counting factors of size $(k,l)$ in a given $f_{m,n}$, contrary to the availability of all $(k+1)$ vertical factors and $(l+1)$ horizontal factors in $f_{\infty,\infty}$, not all will be available in $f_{m,n}$. This is due to the finite nature of $f_{m,n}$. But, by theorem \ref{factcomplex1d}, we know the exact number of horizontal and vertical factors available in any given $f_{m,n}$. Hence we have the following result.

\begin{Proposition}
Given a finite Fibonacci array, $f_{m,n}$, the number of factors of size $(k,l)$ in it are, 
$$ p_{k,l}(f_{m,n}) = p_k(f_m) p_l(f_n).$$
\end{Proposition}
\begin{proof}
Using the count derived in theorem \ref{factcomplex1d}, one can show the equality in identical lines with the proof of proposition \ref{2dfactorcountprop}.
\end{proof}

In Example \ref{2dfactorexamplefinite} also, we count the factors of size $(2,3)$, but in $f_{3,4}$. We can observe the non availability of some horizontal, vertical factors due to the finite nature of $f_{3,4}$.

\begin{Example} \label{2dfactorexamplefinite}
Consider $f_{3,4} = \begin{tabular}{|c c c c c|}
\hline
 d & c & d & d & c  \\
 b & a & b & b & a   \\
 d & c & d & d & c   \\
\hline
\end{tabular}$. 

\vspace{0.5cm}
$\begin{tabular}{|c|| c| c| c|}
\hline
 Horizontal & dcd & cdd & ddc  \\
 Factors of length $3$ $\rightarrow$  & (or) &  (or) & (or)   \\
 $ Vertical $ & bab  &  abb  &  bba      \\
  Factors of length 2 $\downarrow$ &   &   &       \\
  \hline
  \hline
  &&&\\
  
$\begin{tabular}{ccc}
 \begin{tabular}{c}
      d \\
      b 
 \end{tabular} & (or) & \begin{tabular}{c}
      c\\
      a 
 \end{tabular}  \\
\end{tabular}$ &  $\begin{tabular}{|ccc|}

\hline
d&c&d \\
 b&a&b\\
\hline
\end{tabular}$ &   $\begin{tabular}{|ccc|}

\hline
 c&d&d \\
 a&b&b\\
\hline
\end{tabular}$ &  $\begin{tabular}{|ccc|}

\hline
 d&d&c \\
 b&b&a\\
\hline
\end{tabular}$ \\
&&&\\
 \hline
 
  &&&\\
  
$\begin{tabular}{ccc}
 \begin{tabular}{c}
      b \\
      d 
 \end{tabular} & (or) & \begin{tabular}{c}
      a\\
      c 
 \end{tabular}  \\
\end{tabular}$ &  $\begin{tabular}{|ccc|}

\hline
b&a&b\\
d&c&d \\
 
\hline
\end{tabular}$ &   $\begin{tabular}{|ccc|}

\hline
 a&b&b\\
 c&d&d \\

\hline
\end{tabular}$ &  $\begin{tabular}{|ccc|}

\hline
b&b&a\\
d&d&c \\
 
\hline
\end{tabular}$  \\
&&&\\
 \hline
  \end{tabular}$
  
\vspace{0.5cm}
We have, $p_2(f_3) \times p_3(f_4) = 2 \times 3 = 6$ factors of size $(2,3)$.
\end{Example}

\section{2D Fibonacci Words as \texorpdfstring{$S-automatic$}{} Words}\label{DFAOsec}
In this section we construct a DFAO which generates $f_{\infty,\infty}$.

To construct the DFAO, first we construct $|\Sigma|$ number of automata, one for each letter of $\Sigma$, as outlined \cite{Charlier:2010}.These automata will be integrated to get the required DFAO.
\begin{Definition}\cite{Charlier:2010} \label{mnthsym2DS}
For each d-dimensional morphism $\mu: \Sigma \rightarrow B_d{(\Sigma)}$ and for each letter $a \in \Sigma$, define a DFA $\mathscr{A}_{\mu,a}$ over the
alphabet $\{0,1,\ldots,r_{\mu}-1\}^d$ where $r_\mu := max\{|\mu(b)|_i ~~| ~b \in \Sigma, i =1,\ldots,d\}$. The set of states is $\Sigma$, the initial state is $a$ and all states are final. The (partial) transition function is defined by
$$\delta_{\mu}(b,\textbf{n}) = (\mu(b))_{\textbf{n}},~~~ \forall b \in \Sigma ~\text{and}~ \textbf{n} \le |\mu(b)|.$$
\end{Definition}
This automaton will be such that, for all $m,n \ge 0$,
$$ y_{m,n} = \delta_{\mu}(a,(rep_S(m),rep_S(n))^0),$$
where we have padded the shortest word with enough $0$s to make the length of the two words the same. If we consider the coding $v: \Sigma^* \rightarrow \Gamma^*$, as the output function, the corresponding DFAO generates $x$ as an S-automatic sequence. 

The procedure is outlined below.
\begin{itemize}
    \item[Step 1] From the first rows of the morphism $\mu$, derive the one dimensional morphism $\mu_1: \Sigma_1 \rightarrow \Sigma_1^*$ , $\Sigma_1 \subseteq \Sigma$ , which is prolongable on $a$ (This is the restricted morphism along the first direction) 
    \item[Step 2]  Construct the automaton $\mathscr{A}_{\mu_1,a}$ and obtain the directive language  $L{\mu_1,a}$ 
    
    ( Alternatively, we can consider the first columns of the morphism $\mu$ to derive the one dimensional morphism $\mu_2: \Sigma_2 \rightarrow \Sigma_2^*$ , $\Sigma_2 \subseteq \Sigma$, which is also prolongable on $a$. Note that, this is the restricted morphism along the second direction. Then by constructing the automaton $\mathscr{A}_{\mu_2,a}$ we can obtain the directive language  $L{\mu_2,a}$. But since $Shape_{\mu_1}(x) = Shape_{\mu_2}(y)$, where x = $\mu_1^{\omega}(a)$ and $y = \mu_2^{\omega}(a)$, the languages $L_{\mu_1,a}$ and $L_{\mu_2,a}$ will be equal \cite{Charlier:2010})
    \item[Step 3] For each letter $a \in \Sigma$, define a DFA $\mathscr{A}_{\mu,a}$ as described in the Definition \ref{mnthsym2DS}. The DFAO $\mathscr{A}_{\mu}$ is constructed by combining (superimposing) these DFAs. 
     \item[Step 4] For an input $(rep_S(m),rep_S(n))^0$, the output of $\mathscr{A}_{\mu}$ is the symbol $s \in \Sigma$ in the accepting state which will be written at $x_{m,n}$, the $(m,n)^{th}$ entry in the 2D infinite word. 
\end{itemize}

\subsection{DFAO Generating \texorpdfstring{$f_{\infty, \infty}$}{} as an S-automatic word}
As explained earlier, we have the two unidimensional morphisms derived from (2).
Let $\Sigma_1 = \{d,c\}$ and $\Sigma_2 = \{d,b\}$. Then,

$\mu_1: \Sigma_1 \rightarrow \Sigma_1^* ~~:= d \rightarrow d c,~~ c \rightarrow d$ ~~~and ~~$\mu_2: \Sigma_2 \rightarrow \Sigma_2^* ~~:= d \rightarrow d b,~~ b \rightarrow d$.

The automaton $\mathscr{A}_{\mu_1,d}$ will be as in Figure \ref{2Dauto}. And the directive language will be 
$$L_{\mu_1,d} =\{\epsilon, 1, 10, 100, 101, 1000, 1001, 1010, \dots \}.$$

\begin{figure}[ht]
    \centering
    \includegraphics[scale=0.35]{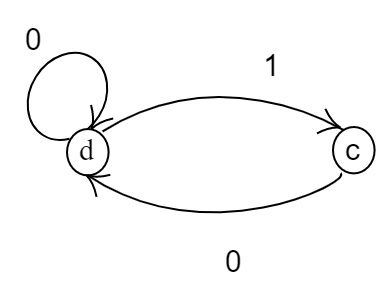}
    \caption{Automaton $\mathscr{A}_{\mu_1,d}$}
    \label{2Dauto}
\end{figure}

Note that $r_\mu := max\{|\mu(b)|_i ~~| ~b \in \Sigma, i =1,2\} = 2$. Hence, for each letter $a \in \Sigma$, we define a DFA, $\mathscr{A}_{\mu,a}$ over the alphabet $\{0,1\}^2 = \{(0,0),(0,1),(1,0),(1,1)\}$. All the four automata are given in Figure \ref{autoall}. We combine all the four automata to get the required automaton, $\mathscr{A}_{\mu}$ , given in Figure \ref{auto}.

\begin{figure}
    \centering
    \includegraphics[scale=0.4]{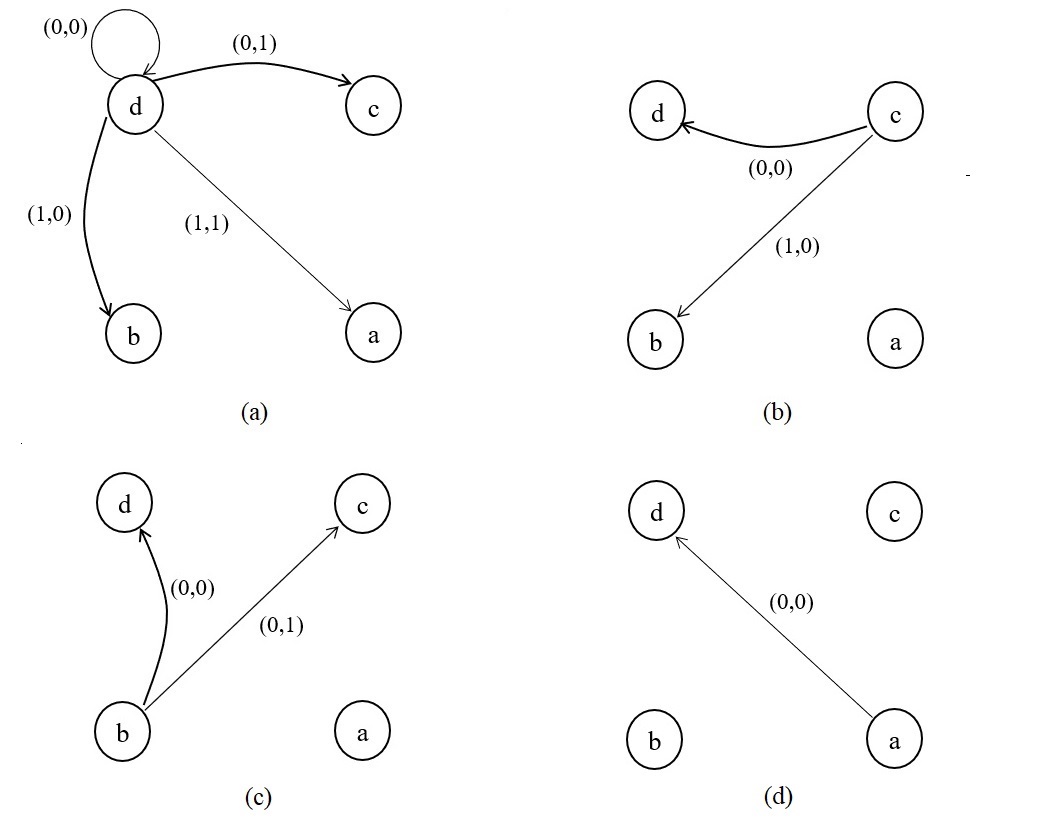}
    \caption{$\mathscr{A}_{\mu,d}$, $\mathscr{A}_{\mu,c}$ , $\mathscr{A}_{\mu,b}$ , $\mathscr{A}_{\mu,a}$ }
    \label{autoall}
\end{figure}

\begin{figure}
    \centering
    \includegraphics[scale=0.6]{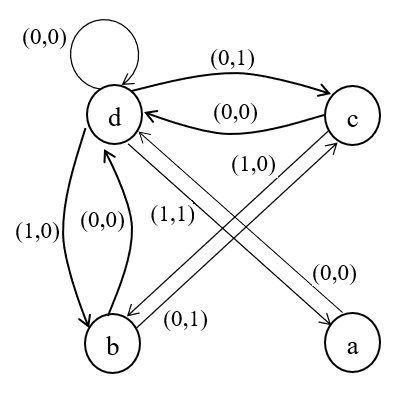}
    \caption{DFAO $\mathscr{A}_{\mu}$ generating $f_{\infty, \infty}$}
    \label{auto}
\end{figure}

For a given $m,n \ge 0$, the steps involved in obtaining the symbol at $(m,n)$ in $f_{\infty, \infty}$ are briefly explained in Example \ref{symbolgener}, using the values $m=2$ and $n=4$.. 
\begin{Example}\label{symbolgener}
Let us generate the symbol at $(2,4)$ in $f_{\infty, \infty}$. From the directive language we get $rep_S(2) = 10$ and $rep_S(4) = 101$.

The $(2,4)^{th}$ entry will be $\delta_{\mu}(d,(10,101)^0) = \delta_{\mu}(d,(010,101))$.
$$d \xrightarrow{(0,1)} c \xrightarrow{(1,0)} b \xrightarrow{(0,1)}  c.$$
\end{Example}
By generating all the symbols at $(i,j)$, $i \in \llbracket 0,s-1 \rrbracket$, $j \in \llbracket 0,t-1 \rrbracket$ the prefix of size, (s,t) of $f_{\infty, \infty}$ can be generated. More specifically any finite 2D Fibonacci array can be generated.

\section{Conclusion}
Though the theory of two-dimensional words is a natural extension of the theory of one-dimensional words, exploring their combinatorial and structural properties is not a straightforward task. In  this paper we have analysed two important properties - tandem repeats and factor complexity - of finite 2D Fibonacci words and the infinite 2D Fibonacci word. A 2D homomorphism with the infinite 2D Fibonacci word as its fixed point is presented, proving that the infinite 2D Fibonacci word is a morphic word. A DFAO generating the 2D Fibonacci words is also constructed. Some closely related properties to tandem repeats are, approximate tandem repeats and approximate periodicity of 2D Fibonacci words. Future research might concentrate on these properties.

\bibliography{mybibfile}

\end{document}